\documentclass{amsart}

\usepackage{graphicx,color}
\newtheorem{thm}{Theorem}[section]

\theoremstyle{definition}

\theoremstyle{remark}

\numberwithin{equation}{section}

\newcommand{\n}{\nabla}

\begin{document}

\title[]
 {On the uniqueness for the heat equation on complete Riemannian manifolds}

\author{Fei He$^1$}

\address{School of Mathematical Science, Xiamen University, 422 S. Siming Rd. Xiamen, Fujian 361000, P. R. China}

\email{hefei@xmu.edu.cn}

\author{Man-Chun Lee$^2$}

\address{Mathematics Department, Northwestern University, 2033 Sheridan Road, Evanston, IL 60208, U.S.A.}

\email{mclee@math.northwestern.edu}

\thanks{$^1$This research is partially supported by the National Natural Science Foundation of China (No.11801474), Fundamental Research
Funds for the Central Universities (No.20720180007) and Natural Science Foundation of Fujian Province (No.2019J05011).}
\thanks{$^2$This research is partially supported by NSF grant DMS-1709894.}
\thanks{}



\date{}

\dedicatory{}

\commby{}


\begin{abstract}
We prove some uniqueness result for solutions to the heat equation on Riemannian manifolds. In particular, we prove the uniqueness of $L^p$ solutions with $0< p< 1$, and improves the $L^1$ uniqueness result of P. Li \cite{Li1984} by weakening the curvature assumption.  
\end{abstract}

\maketitle

\section{Introduction}
In this article we consider the uniqueness for solutions to the heat equation on complete Riemannian manifolds:
\[(\partial_t -\Delta )f = 0,\]
where $\Delta$ is the Laplace-Beltrami operator w.r.t. the Riemannian metric. 

  It is well-known that uniqueness may fail unless we restrict solutions to a suitable class of functions. One example of such a good class that guarantees uniqueness is the set of functions bounded from below or above. 
Recall that uniqueness for nonnegative solutions to the heat equation has been established  in \cite{LY1986} under the so called quadratic Ricci lower bound assumption
\begin{equation}\label{quadratic Ricci lower bound}
Ric(x) \geq - C(r(x)+1)^2,
\end{equation}
where $r(x)$ is the geodesic distance from some fixed point, $C$ is a nonnegative constant. 
 
 Another typical uniqueness class is the set of functions with certain growth rate in the spirite of \cite{T1935}.  For solutions with $L^2$ integrals on geodesic balls or parabolic cylinders growing under certain rate, the uniqueness was proved in \cite{KL} and \cite{G1987}. The same result holds if $L^2$ is replaced by $L^p$ with $1< p \leq 2$, and for a special class of manifolds when $p=1$ \cite{P2015}. These results also imply uniqueness for solutions with suitable pointwise growth rate,  provided that the manifold has some volume growth constraint. A case of particular interest is for bounded solutions, see \cite{G1999} for a survey. 
 
 Our first theorem is an improvement of corresponding results in \cite{KL} and \cite{G1987}. Namely, we allow the integral to be weighted by a positive power of the time variable. An example will be described in section \ref{example}. 
\begin{thm}\label{theorem1}
Let $M$ be a complete Riemannian manifold, and let $f(x,t)$ be a nonnegative subsolution to the heat equation on $M \times (0,1]$ with initial data $f(x, 0) = 0$ in the sense of $L^2_{loc}(M)$. Suppose for some point $q\in M$, and constant $a>0$, 
\[
\int_0^1 t^a \int_{B_q(r)} f^2 \leq e^{L(r)}, \quad \forall r > 0,  
\] 
where $L(r)$ is a positive nondecreasing function satisfying 
\begin{equation}\label{constraint on growth rate}
\int_1^\infty \frac{r}{L(r)} dr = \infty.
\end{equation}
Then $f \equiv 0$ on $M \times (0,1]$.
\end{thm}

In \cite{Li1984}, P. Li considered the uniqueness for $L^p$ solutions to the heat equation. When $p > 1$, the uniqueness holds without further assumption. However when $p=1$ the uniqueness may fail on sufficiently negatively curved manifolds. It was proved in \cite{Li1984} that the uniqueness for $L^1$ solutions holds under the assumption (\ref{quadratic Ricci lower bound}).

As an application of Theorem \ref{theorem1}, we prove the following theorem which can be applied to improve the $L^1$ uniqueness result for the heat equation in \cite{Li1984}. It also implies uniqueness of $L^p$ solutions with $0<p< 1$. The curvature assumption (\ref{super quadratic Ricci lower bound}) (\ref{divergence assumption on lower bound function}) is slightly more general than (\ref{quadratic Ricci lower bound}) since functions such as $r \ln r$ are allowed. 
\begin{thm}\label{theorem2}
Let $(M^n,g)$ be a complete Riemannian manifold with
\begin{equation}\label{super quadratic Ricci lower bound} 
Ric(x) \geq - k^2(r_q(x)),
\end{equation}
where $r_p(x)$ is the distance function to a fixed point $q\in M$, and $k(r)$ is a positive nondecreasing function satisfying
\begin{equation}\label{divergence assumption on lower bound function}
\int_1^\infty \frac{1}{k(r)} dr = \infty. 
\end{equation}
 Suppose $f$ is a nonnegative subsolution to the heat equation on $M \times (0,1]$, with initial data $f( 0) = 0$ in sense of $L^2_{loc}(M)$. If for some $0< p \leq 1$,
\[
\|f\|_{L^p(B_q(r)\times[0, 1])} \leq e^{Crk(r)}, \quad  \text{for any} \quad r >0,
\]
for some constant $C$, 
then $f \equiv 0$ on $M \times (0,1]$. 
\end{thm}
For the proof of Theorem \ref{theorem2}, we use mean value inequality to get a pointwise bound for the solution, which is non-uniform and blows up as $t\to 0$, and we can verify that the assumptions in Theorem \ref{theorem1} are satisfied. 

To prove the uniqueness of solutions to the heat equation, we can consider a solution starting with $0$ initial data and apply Theorem \ref{theorem1} or \ref{theorem2} to its absolute value which is a nonnegative subsolution. 

Also, it is well-known that results like the theorems above can be used as a maximum principle. Suppose $u$ is a subsolution to the heat equation with $u(0)\leq 0$, one can apply Theorem \ref{theorem1} or \ref{theorem2} to $(u-0)_+$ to show that $u(t) \leq 0$ provided the assumptions are met.

\section{Proof}

\begin{proof}[Proof of Theorem \ref{theorem1}]
The proof is a modification of the arguments due to Karp-Li \cite{KL} and Grigor\rq{}yan \cite{G1987}. 
Define the function
 \[
 \xi(x,t) = - \frac{(r(x) - R)_+^2}{4(T - t)},
 \] 
 where $r(x)$ is the distance function to the fixed point $q$, then it is a direct calculation to check 
\[
\partial_t \xi + |\n \xi|^2 \leq 0.
\]
For any $R>0$, let $\psi(r)$ be a nonincreasing cut-off function such that 
\[
\psi(r) = \begin{cases}
1, & r \leq \frac{3}{2}R; \\
0, & r > 2R.
\end{cases}
\]
And let $\phi (x) = \psi^m(r(x))$ for some $m>0$ to be chosen later, then
\[
|\n \phi|^2 \leq \frac{4m}{R^2} \phi^{2-2/m}.
\]
For any $0 < \tau < T \leq 1$, we can calculate
\[
\begin{split}
&\int_\tau^T t^{-1} \int \phi^2 e^\xi f^2 \partial_t \xi + 2 \phi^2 e^\xi f \partial_t f \\
\leq & \int_\tau^T t^{-1} \int \phi^2 e^\xi f^2 \partial_t \xi + 2  \phi^2 e^\xi f \Delta f\\
= & \int_\tau^T t^{-1} \int \phi^2 e^\xi f^2 \partial_t \xi -2  \phi^2 e^\xi f \langle \n \xi, \n f\rangle- 2  \phi^2 e^\xi |\n f|^2 - 4 \phi e^\xi f \langle \n \phi, \n f\rangle \\
\leq & \int_\tau^T t^{-1} \int \phi^2 e^\xi f^2 (\partial_t \xi + |\n \xi|^2 ) + 4  |\n \phi|^2 e^\xi f^2 \\
\leq &  4 \int_\tau^T t^{-1} \int    |\n \phi|^2 e^\xi f^2.
\end{split}
\]
By the choice of $\phi$, using similar arguments as in \cite{LM2019}, we can use the Holder inequality and Young's inequality to show
\[
\begin{split}
\int    |\n \phi|^2 e^\xi f^2 \leq & \frac{4m}{R^2} \int_{spt(\n \phi)} \phi^{2-2/m} e^\xi f^2 \\
\leq & \frac{4m}{R^2} \left( \int_{spt(\n \phi)} \phi^2 e^\xi f^2 \right)^{(1-1/m)} \left( \int_{spt(\n \phi)} e^\xi f^2 \right)^\frac{1}{m}  \\
\leq & \frac{1}{4 t} \left( \int_{spt(\n \phi)} \phi^2 e^\xi f^2 \right) + \frac{C(m)t^{m-1}}{R^{2m}}  \left( \int_{spt(\n \phi)} e^\xi f^2 \right),
\end{split}
\]
where $C(m)=4^{2m-1}m^m$. Plug in to the previous inequality, we get
{ \[
\begin{split}
&\frac{1}{T} \int \phi^2 e^\xi f^2 (T) - \frac{1}{\tau} \int \phi^2 e^\xi f^2 (\tau) \\
= & \int_\tau^T t^{-1} \int \phi^2 e^\xi f^2 \partial_t \xi + 2 \phi^2 e^\xi f \partial_t f - \int_\tau^T t^{-2} \left( \int \phi^2 e^\xi f^2 \right)\\
\leq & \frac{C(m)}{R^{2m}}  \int_\tau^T t^{m-1}\int_{spt(\n \phi)} e^\xi f^2 .
\end{split}
\]}
On $spt(\n \phi) \subset B_q(2R) \backslash B_q(\frac{3}{2}R)$ we have 
\[
\xi \leq {- \frac{R^2}{16(T-\tau)}}, 
\]
hence if we take $m > a + 1$, then the growth assumption on $f$ implies
\[
\int_\tau^T t^{m-1}\int_{spt(\n \phi)} e^\xi f^2 \leq T^{m-a-1}e^{- \frac{R^2}{16(T-\tau)} + L(2R)}. 
\]
Now if we require that
\[
(T-\tau) \leq \frac{R^2 }{16L(2R) },
\]
then 
\begin{equation}\label{inductive inequality}
\frac{1}{T} \int_{B(R)} f^2 (T) - \frac{1}{\tau} \int_{B(2R)} f^2 (\tau) \leq \frac{C(m) T^{m-a-1}}{R^{2m}}.
\end{equation}
To proceed we take an increasing sequence of $R_i$, and a decreasing sequence of $\tau_i$ in the following way. Let $R_i = 2^i R$, $\tau_0 = \tau$ and take $\tau_{i+1}$ such that 
\[
\tau_i - \frac{R_i^2}{16 L(2R_i)} \leq \tau_{i+1} < \tau_i. 
\]
Then for any $N$, using (\ref{inductive inequality}) inductively we have
\begin{equation}\label{vanishing estimate}
\begin{split}
& \frac{1}{\tau} \int_{B(R)} f^2 (\tau) \\
= & \frac{1}{\tau_N} \int_{B(R_N)} f^2 (\tau_N ) + 
 \sum_{i=0}^{N-1} \left( \frac{1}{\tau_i} \int_{B(R_i)} f^2 (\tau_i) - \frac{1}{\tau_{i+1}} \int_{B(R_{i+1})} f^2 (\tau_{i+1}) \right) \\
\leq &  \frac{1}{\tau_N} \int_{B(R_N)} f^2 (\tau_N ) + 
 \sum_{i=0}^{N-1} \frac{C(m) \tau_i^{m-1-a}}{R_i^{2m}}\\
 \leq &  \frac{1}{\tau_N} \int_{B(R_N)} f^2 (\tau_N ) + \frac{2C(m)\tau^{m-1-a}}{R^{2m}}.
\end{split}
\end{equation}
By the assumption on $L(r)$ (\ref{constraint on growth rate}), we must have
\[
\sum_{i=0}^\infty \frac{R_i^2}{l(R_i)} = \infty,
\]
hence we can choose the sequence $\{\tau_i\}$ such that $\tau_i$ becomes zero in finite steps. 

To show that the first term in the last line of (\ref{vanishing estimate}) can be dropped,
we claim that for any $R>0$, we have
\[
\lim_{t \to 0^+} \frac{1}{t} \int_{B(R)} f^2(t)= 0.
\]
Now we prove this claim. For any cut-off function $\phi$, since $\lim_{t\to 0^+}\int \phi^2 f^2 (t) = 0$, we have 
\[
\begin{split}
0 \geq &  \int_0^t \int \phi^2 f (\partial_t f -  \Delta f) \\
= & \frac{1}{2}\int \phi^2 f^2(t)  + \int_0^t \int \phi^2 |\n f|^2 + 2 \int _0^t \int \phi f \langle \n \phi, \n f\rangle \\
\geq & \frac{1}{2}\int \phi^2 f^2(t)  - \int_0^t \int |\n \phi|^2 f^2.
\end{split}
\]
Choose a cut-off function $\phi$ similarly as before such that $|\n \phi| \leq C \phi^{1-1/m}$ for some $m \geq 2$, then the above inequality yields
\[
\begin{split}
\int \phi^2 f^2(t) \leq & C \int_0^t \int (\phi^2 f^2 )^\frac{m-1}{m} f^\frac{2}{m} \\
\leq & C \int_0^t \left( \int \phi^2 f^2 \right)^\frac{m-1}{m} \left(\int_{spt(\phi)} f^2 \right)^\frac{1}{m} \\
\leq & C \sup_{s\in(0,t)}\|f(s)\|_{L^2(spt(\phi))}^\frac{1}{m}\int_0^t \left( \int \phi^2 f^2 \right)^\frac{m-1}{m} \\
\leq &  C  \sup_{s\in(0,t)}\|f(s)\|_{L^2(spt(\phi))}^\frac{1}{m} \left( \int_0^t \int \phi^2 f^2\right)^\frac{m-1}{m} t^m,
\end{split}
\]
by the assumption on $f$ we see that the RHS is of the order $o(t^m)$. Since the cut-off function $\phi$ is chosen for an arbitrary radius, the claim is now proved. 

Therefore we only need to let $R \to \infty$ in (\ref{vanishing estimate}) to show that $f(\tau) \equiv 0$ for any $\tau \in(0,1]$.
\end{proof}

\begin{proof}[Proof of Theorem \ref{theorem2}]
By \cite{Sa1992}, the curvature assumption (\ref{super quadratic Ricci lower bound}) implies that there is a Sobolev inequality in the form
\[
\left( \int \phi^{\frac{2n}{n-2}} \right)^\frac{n-2}{n} \leq \frac{R^2 e^{C(n) Rk(R)}}{Vol(B_q(R))^\frac{2}{n}} \int (|\n \phi|^2 + R^{-2} \phi^2 ), 
\]
for any smooth function $\phi$ compactly supported in the geodesic ball $B_q(R)$. With this Sobolev inequality, we can apply Nash-Moser iteration to prove a mean value inequality for $f$ (see Chapter 19 of \cite{Li2012}), for any $t\in(0,1]$,  
\[
\|f\|_{L^\infty(B_q(\frac{R}{2})\times [\frac{t}{2},t] )} \leq C(n, p) e^{C\alpha Rk(R)} ( \frac{1}{R^{\beta}} + \frac{1}{t})^{\gamma} Vol (B_q(\frac{R}{2}))^{-\frac{1}{p}} \|f\|_{L^p(B_q(R) \times [0,t])},
\]
where $\alpha, \beta, \gamma$ are positive constants depending on $n$ and $p$. Without loss of generality we can take $R > 2$ hence
\[Vol(B_q(R/2)) \geq v_0 : = Vol(B_q(1)).\]
Now the assumption of the theorem implies
\[
|f(x, t)| \leq \frac{e^{Cr(x)k(2r(x)) }}{t^a}, \quad t\in(0,1],
\]
for some constants $C$ depending on $n, p, v_0$, and the constant $a$ only depends on $n$ and $p$. 
By (\ref{quadratic Ricci lower bound}) and volume comparison theorem, we have the volume growth estimate
\[Vol(B_q(R)) \leq C(n) e^{c(n) Rk(R)},\]
for any $R > 0$, see for example \cite{G1999}. Hence we can verify that
\[ 
\int_0^1 t^a \int_{B(R)} f^2 \leq e^{{L}(R)},
\]
with 
\[
{L}(R) = CRk(2R) ,
\]
for some constant $C$. By (\ref{divergence assumption on lower bound function}), the function $L(R)$ satisfies (\ref{constraint on growth rate}). Then we can apply Theorem \ref{theorem1} to finish the proof. 
\end{proof}

\section{Example}\label{example}
In this section we describe the construction of a solution to the heat equation, which belongs to the uniqueness class of Theorem \ref{theorem1}, but not in that of \cite{KL} \cite{G1987} or \cite{P2015}. Intuitively, we want to construct a solution whcih has a sequence of `spikes' with fast growing heights, while supported on decaying domains such that we have some integral control of the solution locally.

Let's take $M = \mathbb{R}^n$ with $n \geq 3$, and we will make several assumptions for simplicity, however the same method can be used to construct more complicated examples. 

To start with, let $\tilde{u}_0$ be a continuous function on $\mathbb{R}^n$ with growth rate slower than $e^{C|x|^2}$. For simplicity we take $\tilde{u}_0 \geq 0$ and 
$\tilde{u}_0 \in L^1(\mathbb{R}^n)$.

We can construct a ``spiked'' initial function $u_0$ by modifying $\tilde{u}_0$:  for each positive integer $i = 1, 2, 3 ,...$, choose a geodesic ball
\[ B(p_i, r_i) \subset B(0, i+1)\backslash B(0, i),\]
where we take the radii to be
\[r_i = \left( \frac{1}{\omega_n i^2e^{i^3}} \right)^\frac{1}{n},\]
$\omega_n$ is the volume of the unit ball in $\mathbb{R}^n$. Denote 
\[
\tilde{r}_i = \frac{r_i}{2^{1/n}}.
\]
Modify $\tilde{u}_0$ in each $B(p_i, r_i)$ to get a new function
\[
u_0 = \begin{cases}
 e^{i^3}, & \text{on} \quad B(p_i, \tilde{r}_i), \\
\text{continuous and $\leq  e^{i^3}$}, & \text{on} \quad B(p_i, r_i)\backslash B(p_i, \tilde{r}_i), \\
\tilde{u}_0, & \text{otherwise.}
\end{cases}
\]
The new function $u_0$ is a countinuous function which is $L^1$ on the modified region $\cup_{i = 1}^\infty B(p_i, r_i)$.

 Solve the Cauchy problem of the heat equation with initial function $u_0$ by convoluting with the heat kernel: 
\[
u(x,t) = \int \frac{1}{(4 \pi t)^{n/2}} e^{-\frac{|x-y|^2}{4t}} u_0(y) dy.
\]
For each $x \in B(p_i, \tilde{r}_i)$ and $t> 0$,
\[
\begin{split}
u(x,t) \geq & \frac{1}{(4\pi t)^{n/2}} \int_{B(p_i, \tilde{r}_i)} e^{-\frac{|x-y|^2}{4t}} u_0(y) dy \\
\geq & \frac{1}{(4\pi t)^{n/2}} e^{-\frac{4\tilde{r}_i^2}{4t} } e^{i^3} \omega_n \tilde{r}_i^n \\
= &   \frac{1}{2i^2 (4\pi t)^{n/2}} e^{-\frac{\tilde{r}_i^2}{t} } .
\end{split}
\]
Hence
\[
\begin{split}
\int_0^1 \int_{B(0, i+1)} u(x,t)^2 \geq \int_0^1 \int_{B(p_i, \tilde{r}_i)} u(x,t)^2 \geq & \frac{\omega_n \tilde{r}_i^n}{4 i^4 }\int_0^1 (4 \pi t)^{-n} e^{- 2\tilde{r}_i^2 / t} dt \\
= & \frac{\omega_n}{2^{n+1} (4 \pi)^n} \frac{1}{i^2 \tilde{r}_i^{n-2}} \int_1^\infty s^{n-2} e^{-s} ds \\
\geq & C(n) e^{\frac{n-2}{n} i^3}.
\end{split}
\]
Thus $u$ violates the assumption in either \cite{KL} or \cite{G1987} when $n\geq 3$. For $L^p$ integrals with $p>1$ one can compute similarly. 

On the other hand, since we assumed $u_0$ to be $L^1$, we have
\[
| u(x, t) | \leq \frac{\|u_0\|_{L^1(\mathbb{R}^n)}}{(4\pi t)^{n/2}},
\]
therefore it satisfies the assumption of Theorem \ref{theorem1}.

To construct examples which are not in $L^1$, we can start with $\tilde{u}\equiv 1$ instead of a $L^1$ function; and to construct examples not bounded from either side, we can add a sequence of \lq\lq{} negative spikes\rq\rq{} to $u_0$ sufficiently far away from the positive one. 

\textbf{Acknowledgements:} The first named author woule like to thank Prof. Jiaping Wang for helpful communication.


\begin{thebibliography}{XXXXX9}
\bibitem[G87]{G1987} A. Grigor'yan, {\sl On stochastically complete manifolds,} Soviet Math. Dokl. 34 (1987), 310-313.
\bibitem[G99]{G1999} A. Grigor'yan, {\sl Analytic and geometric background of recurrence and non-explosion of the Brownian motion on Riemannian manifolds},  Bull. Amer. Math. Soc. (N.S.) 36 (1999), no. 2, 135–249. 
\bibitem[KL]{KL} L. Karp and P. Li, {\sl The heat equation on complete Riemannian manifolds}, unpublished.
\bibitem[LM19]{LM2019} M.-C. Lee and J. M.-S. Ma, {\sl Uniqueness Theorem for non-compact mean curvature flow with possibly unbounded curvatures}, to appear on Comm. Anal. Geom. arXiv:1709.00253.
\bibitem[Li84]{Li1984} P. Li, {\sl Uniqueness of $L^1$ solutions for the Laplace equation and the heat equation on Riemannian manifolds}, J. Differential Geometry, 20 (1984) 447-457.
\bibitem[Li12]{Li2012} P. Li, {\sl Geometric Analysis}, Cambridge Studies in Advanced Mathematics, vol.134, Cambridge University Press, Cambridge, ISBN978-1-107-02064-1, (2012), x+406 pp.
\bibitem[LY86]{LY1986}P. Li and S.-T. Yau, {\sl On the parabolic kernel of the Schrodinger operator}, Acta Math. 156 (1986), no. 3-4, 153–201.
\bibitem[P15]{P2015} F. Punzo, {\sl Uniqueness for the heat equation in Riemannian manifolds}, J. Math. Anal. Appl. 424 (2015), no. 1, 402–422.
\bibitem[Sal92]{Sa1992} L. Sallof-Coste, {\sl Uniformly elliptic operators on Riemannian manifolds}, Journal of Differential Geometry, 36, (1992), 417-450.
\bibitem[T35]{T1935} A. N. Tichonov, {\sl Uniqueness theorems for the equation of heat conduction}, (in Russian) Matem. Sbornik, 42 (1935) 199-215.
\end{thebibliography}
\end{document}